\documentclass{amsart}
\usepackage{amsmath}
\usepackage{amsfonts}
\usepackage{amsthm}
\usepackage{amssymb}
\usepackage{cite}
\usepackage{enumerate}
\usepackage[all]{xy}
\usepackage{cite}
\usepackage{mathrsfs}
\usepackage{color}
\usepackage{xcolor}
\usepackage{hyperref}
\usepackage{fancyhdr}
\hypersetup{
	colorlinks=true,
	anchorcolor=blue,
	linkcolor=blue,
	filecolor=blue,
	urlcolor=blue,	
	citecolor=blue,
	bookmarks=true,
	bookmarksopen=true,
	pdfborder=000
}
\numberwithin{equation}{section}
\theoremstyle{plain}
\newtheorem{thm}{Theorem}[section]

\newtheorem{proposition}[thm]{Proposition}

\newtheorem{cor}[thm]{Corollary}

\newtheorem{lemma}[thm]{Lemma}
\newtheoremstyle{noparens}%
 {}{}%
 {\itshape}{}%
 {\bfseries}{.}%
 { }%
 {\thmname{#1}\thmnumber{ #2}\mdseries\thmnote{ #3}}
\theoremstyle{noparens}
\newtheorem{lemmaNoParens}[thm]{Lemma}

\theoremstyle{definition}
\newtheorem{defn}[thm]{Definition}
\theoremstyle{remark}
\newtheorem{rmk}[thm]{Remark}

\makeatletter

\newcommand{\C}{Carath$\acute{\operatorname{e}}$odory}
\newcommand{\Rmnum}[1]{\expandafter\@slowromancap\romannumeral #1@}
\newcommand{\li}{\lim\limits_{z\rightarrow\partial\Omega}}

\newcommand{\K}{K$\ddot{\operatorname{a}}$hler}
\makeatother
\begin{document}
\title{K$\ddot{\operatorname{A}}$hlerity of invariant metrics on pseudoconvex domain of dimension two}

\author{Lang Wang\textsuperscript{1}}
\address{$1.$ School of Mathematical Sciences, Guizhou Normal University, Guiyang, 550025, P.R. China.}
\email{wanglang2020@amss.ac.cn}

\begin{abstract}

For a two dimensional bounded pseudoconvex domain of finite type, we prove uniformization theorems via \K-Kobayashi metric or \K-\C\ metric with quasi-finite geometry of order three. In particular, a pseudoconvex Reinhardt domain of finite type  is the unit ball if and only if the Bergman metric is a scalar multiple of the Kobayashi metric or \C\ metric. Moreover, we establish a rigidity theorem concerning holomorphic sectional curvature of Bergman metric and Lu constant.

\end{abstract}

\maketitle

\section{Introduction}
Let $\Omega$ be a bounded pseudoconvex domain in $\mathbb{C}^n$, there are a number of important invariant metrics on it: the Kobayashi metric, the \C\ metric, the Bergman metric and the \K-Einstein metric. All of these metrics coincide, up to a constant, on the unit ball. The following question arises naturally.
\medskip

\noindent{\textbf{Question.}} For which domain, the previous metrics are not coincide?
\medskip

In 1979, Cheng conjectured that if the Bergman metric on a smoothly bounded strongly pseudoconvex domain is \K-Einstein, then it is biholomorphically equivalent to the unit ball. Two dimensional case was proved by Fu-Wong\cite{fw} and Nemirovski\u i-Shafikov\cite{cr}. And this conjecture was completely confirmed by Huang-Xiao. Recently, Savale-Xiao \cite{sx} proved the case of two dimensional pseudoconvex domain of finite type.

\begin{thm}[{\cite[Theorem 1.1]{hx}}]
Suppose $\Omega\subset\mathbb{C}^n$ is a smoothly bounded strongly pseudoconvex domain, then the Bergman metric is K$\ddot{a}$hler-Einstein if and only if it is biholomorphically equivalent to the unit ball.
\end{thm}

As for invariant Finsler metrics, such as Kobayashi metric and \C\ metric, Gaussier-Zimmer firstly proved the following theorem.

\begin{thm}[{\cite[Theorem 1.5]{hartogs}}]\label{zi}
Let $\Omega\subset{\mathbb{C}^n}$ be a bounded strongly pseudoconvex domain with $C^2$ smooth boundary. Then the following are equivalent:

(1) $\Omega$ is biholomorphically equivalent to a ball quotient.

(2) the Kobayashi metric is a  K$\ddot{a}$hler metric,

(3) the Kobayashi metric is a  K$\ddot{a}$hler metric with constant holomorphic sectional curvature.

\end{thm}

Furthermore, Theorem \ref{zi} means that for a bounded strongly pseudoconvex domain with $C^2$ smooth boundary, if the Bergman metric is a scalar multiple of the Kobayashi metric, then it is biholomorphic to the unit ball. After Theorem \ref{zi}, Dong-Wang-Wong\cite{dww} proved a similar version of \C\ metric. Moreover, they also provided a rigidity theorem about holomorphic sectional curvature of Bergman metric with respect to Lu constant.
\medskip

For a bounded domain $\Omega\subset\mathbb{C}^n$, we say that Kobayashi metric  asymptotically  equals \C\ metric at infinity, if $\li \frac{C_{\Omega}(z,v)}{K_{\Omega}(z,v)}=1$ holds uniformly for $0\neq v\in\mathbb{C}^n$.

 Suppose $\Omega\subset\mathbb{C}^n$ is a strongly pseudoconvex domain with $C^2$ smooth boundary, then Kobayashi metric  asymptotically  equals \C\ metric at infinity. Moreover, if $\Omega\subset\mathbb{C}^2$ is a smoothly bounded pseudoconvex domain of finite type, which is biholomorphic to a ball quotient, then \cite[Proposition 1.3]{str} implies that Kobayashi metric  asymptotically  equals \C\ metric at infinity. Since the covering map is a local isometry from the unit ball to $\Omega$ with respect to Kobayashi metric and \K-Einstein metric, then Kobayashi metric on $\Omega$ is a \K\ metric with quasi-finite geometry of order infinity (see section \ref{sec2}). In this paper, we firstly prove the converse result as follows.

\begin{thm}\label{kob}
Let $\Omega$ be a smoothly bounded pseudoconvex domain of finite type in $\mathbb{C}^2$. Then the following are equivalent:

(1)\ $\Omega$ is biholomorphically equivalent to a ball quotient,

(2)\ the Kobayashi metric is a K$\ddot{a}$hler metric with quasi-finite geometry of order three and Kobayashi metric asymptotically equals \C\ metric at infinity,

(3)\ the Kobayashi metric is a K$\ddot{a}$hler metric  with constant holomorphic sectional curvature.

\end{thm}

Since \cite[Proposition 3.11]{str}, then the Bergman metric on $\Omega$ has quasi-finite geometry of order infinity, where $\Omega\subset\mathbb{C}^2$ is a smoothly bounded pseudoconvex domain of finite type. As a corollary, the following holds.

\begin{cor}
Suppose $\Omega\subset\mathbb{C}^2$ is a smoothly bounded domain of finite type, then the following are equivalent:

(1) $\Omega$ is biholomorphic to the unit ball,

(2) the Bergman metric is a scalar multiple of the Kobayashi metric and Kobayashi metric asymptotically equals Carath$\acute{e}$odory metric at infinity.
\end{cor} 

Moreover,  the following uniformazition theorem with respect to K$\ddot{\operatorname{a}}$hler-Einstien metric holds.

\begin{cor}\label{kem}
Suppose $\Omega\subset\mathbb{C}^2$ is a smoothly bounded domain of finite type, then the following are equivalent:

(1) $\Omega$ is biholomorphic to a ball quotient,

(2) the K$\ddot{a}$hler-Einstein metric is a scalar multiple of the Kobayashi metric and Kobayashi metric asymptotically equals Carath$\acute{e}$odory metric at infinity.
\end{cor}

As for \C\ metric, we obtain the following theorem by using a similar proof of Theorem \ref{kob}.

\begin{thm}\label{car}
Let $\Omega$ be a smoothly bounded pseudoconvex domain of finite type in $\mathbb{C}^2$. Then the following are equivalent:

(1)\ $\Omega$ is biholomorphically equivalent to the unit ball,

(2)\ the Carath$\acute{e}$odory metric is a K$\ddot{a}$hler metric with quasi-finite geometry of order three and Kobayashi metric asymptotically equals Carath$\acute{e}$odory metric at infinity,

(3)\ the Carath$\acute{e}$odory metric on $\Omega$ is a K$\ddot{a}$hler metric  with constant holomorphic sectional curvature,

(4) the Bergman metric or K$\ddot{a}$hler-Einstein metric is a scalar multiple of the Carath$\acute{e}$odory metric and Kobayashi metric asymptotically equals Carath$\acute{e}$odory metric at infinity.
\end{thm}

Suppose $\Omega\subset\mathbb{C}^n$ is a bounded domain, and $B_{\Omega}(z,v)$ is Bergman metric on $\Omega$. Then we have $B_{\Omega}(z,v)\geq C_{\Omega}(z,v)$ for each $(z,v)\in\Omega\times\mathbb{C}^n$, where $C_{\Omega}$ is the \C\ metric on $\Omega$. For the optimal control of these metrics, it is natural to consider the Lu constant $L(\Omega)$ (see section \ref{metric}). Combining with the holomorphic sectional curvature of Bergman metric and Lu constant, we obtain the following result.
\begin{thm}\label{bc}
Let $\Omega$ be a smoothly bounded domain of finite type in $\mathbb{C}^2$, then the following are equivalent:

(1) $\Omega$ is biholomorphic to the unit ball,

(2)  the holomorphic sectional curvature of $B_{\Omega}(z,v)$ is bounded above by $-4(L(\Omega))^2$ and Kobayashi metric asymptotically equals Carath$\acute{e}$odory metric at infinity.
\end{thm}

Let  $\Omega$ be a bounded domain in $\mathbb{C}^n$ and $z\in\partial\Omega$ be a boundary point, we say that $\Omega$ is locally convexifiable at $z$, if  there exists a neighbourhood $U$ of $z$ and a biholomorphic mapping $\Phi$ from $U$ into $\mathbb{C}^n$ such that $\Phi(U\cap\Omega)$ is convex.

For a smoothly bounded pseudoconvex Reinhardt domain of finite type in $\mathbb{C}^2$, Theorem 3.2 in \cite{fu} implies that it is locally convexifiable at every boundary point. Combining with the proofs of  \cite[Theorem 1.2]{ni} and \cite[Theorem 1]{loc}, we know that for such domain, Kobayashi metric asymptotically equals \C\ metric at infinity. Thus we obtain the following result.
\begin{thm}
Let $\Omega$ be a smoothly bounded pseudoconvex Reinhardt domain of finite type in $\mathbb{C}^2$, then the following are equivalent:

(1)\ $\Omega$ is biholomorphically equivalent to a ball quotient,

(2)\ the Kobayashi metric is a K$\ddot{a}$hler metric with quasi-finite geometry of order three,

(3) the K$\ddot{a}$hler-Einstein metric is a scalar multiple of Kobayashi metric,

(4)\ the Kobayashi metric is a K$\ddot{a}$hler metric  with constant holomorphic sectional curvature.

Moreover, the following are also equivalent:

(4)\ $\Omega$ is biholomorphically equivalent to the unit ball,

(5)\ the Carath$\acute{e}$odory metric is a K$\ddot{a}$hler metric with quasi-finite geometry of order three,

(6)\ the Carath$\acute{e}$odory metric  is a K$\ddot{a}$hler metric  with constant holomorphic sectional curvature,

(7) the Bergman metric or K$\ddot{a}$hler-Einstein metric is a scalar multiple of the Kobayashi metric or Carath$\acute{e}$odory metric,

(8)\ the holomorphic sectional curvature of Bergman metric is bounded above by $-4(L(\Omega))^2$.
\end{thm}

This paper is orgnized as follows. We give the preliminaries in section \ref{sec2}. Section \ref{sec3} is denoted to prove main theorems.

\section{Preliminaries}\label{sec2}
\subsection{Notations}\hfill

(1) Let $|\cdot|$ be the standard Euclidean norm in $\mathbb{C}^n$.

(2) $\Delta$ will be denoted as the unit disc in $\mathbb{C}$, and $\Delta_r:=\left\{z\in\mathbb{C}:|z|< r\right\}$.  The unit ball in $\mathbb{C}^n$ for $n\ge2$ is denoted by $\mathbb{B}^n$, and $B_r(z_0):=\left\{z\in\mathbb{C}^n:|z-z_0|<r\right\}$.

(3) For any $n\geq1$, we let $(\mathbb{C}^n)^*:=\mathbb{C}^n\setminus\{0\}$.

(4) For $\Omega\varsubsetneq\mathbb{C}^n$, we denote $\delta_{\Omega}(x):=\inf\left\{|z-x|:z\in\partial\Omega\right\}$. 

\subsection{Invariant metric}\label{metric}\hfill

\begin{defn}
Suppose $\Omega$ is a domain in $\mathbb{C}^n$. For an upper semicontinuous function $F:\Omega\times\mathbb{C}^n\rightarrow[0,\infty)$, we say that $F$ is a $Finsler\ metric$ on $\Omega$, if the equality $F(z,tX)=|t|F(z,X)$ holds for any $z\in\Omega,t\in\mathbb{C},X\in\mathbb{C}^n$.
\end{defn}

Let $F$ be a Finsler metric on $\Omega$, it can induce a distance
\begin{align*}
d_F(x,y):=\inf\{L_F(\gamma):&\gamma:[0,1]\rightarrow\Omega\ \operatorname{is}\ \operatorname{a}\ \operatorname{piecewise}\ C^1\  \operatorname{smooth}\\ &\operatorname{curve}\ \operatorname{with}\ \gamma(0)=x\ \operatorname{and}\ \gamma(1)=y\}
\end{align*}
for any $x,y\in\Omega$. Here $L_F(\gamma)$ is defined as $L_F(\gamma):=\int_{0}^1F(\gamma(t),\dot{\gamma}(t))dt$.

For a domain $\Omega\subset\mathbb{C}^n$, the Kobayashi metric on $\Omega$ is defined by
\[
K_{\Omega}(z,X):=\inf\left\{|\xi|:\exists{f\in \operatorname{Hol}(\Delta,\Omega),\ \operatorname{with}\ f(0)=z,d(f)_0(\xi)=X}\right\}
\]
for any $z\in\Omega$ and $0\neq X\in\mathbb{C}^n$. And the \C\ metric is denoted as
\[
C_{\Omega}(z,X):=\operatorname{sup}\left\{|df_z(X)|:f\in\operatorname{Hol}(\Omega,\Delta),f(z)=0\right\}
\]
for $z\in\Omega,0\neq X\in\mathbb{C}^n$.

Suppose 
\[
\kappa_{\Omega}(z,z):=\sup\{|f(z)|^2:f\in\operatorname{Hol}(\Omega,\mathbb{C}),||f||_{L^2(\Omega)}\leq1\}
\]
is the Bergman kernel in the diagonal of $\Omega\times\Omega$. Then the Bergman metric on $\Omega$ is defined by
\[
B_{\Omega}(z,X):=\frac{\sup\{|X(f)(z)|:f\in\operatorname{Hol}(\Omega,\mathbb{C}),f(z)=0,||f||_{L^2(\Omega)}\leq1\}}{\kappa^{1/2}_{\Omega}(z,z)}.
\]

\begin{defn}
Let $\Omega$ be a bounded domain in $\mathbb{C}^n$, then the $Lu\ constant\ L(\Omega)$ is defined as
\[
L(\Omega):=\sup_{z\in\Omega,0\neq v\in T_z(\Omega)}\frac{C_{\Omega}(z,v)}{B_{\Omega}(z,v)}.
\]
\end{defn}

Note that $L(\Omega)\leq 1$ by the definition of Lu constant. Moreover, $L(\Omega)=\frac{1}{\sqrt{n+1}}$ for the unit ball $\mathbb{B}^n$, and $L(\Omega)=\sqrt{2n}$ for the unit polydisk $\Delta^n$. When $\Omega$ is a bounded strongly pseudoconvex domain with $C^2$ smooth boundary in $\mathbb{C}^n$, since $\lim\limits_{z\rightarrow\partial\Omega}\frac{C_{\Omega}(z,v)}{B_{\Omega}(z,v)}=\frac{1}{\sqrt{n+1}}$ holds uniformly for $0\neq v\in\mathbb{C}^n$, then $L(\Omega)\geq\frac{1}{\sqrt{n+1}}$.

\subsection{Holomorphic sectional curvature}\hfill

 Let $(M,J,g)$ be a \K\ manifold, and $R(g)$ be the curvature tensor of $(M,g)$. For any $z\in M$ and non-zero tangent vector $X\in T_zM$, the holomorphic sectional curvature is defined by 
\[
H(g)(X):=\frac{R(X,JX,X,JX)}{g(X,X)^2}.
\]

The above definition relies on the smoothness of metric $g$. For a Finsler metric, we have a similar definition. We firstly recall the Gaussian curvature for pseudo-metric on $\Delta$.

\begin{defn}
Suppose $ds^2=gdz\otimes d\bar{z}$ is a pseudo-metric on $\Delta$, where $g$ is an upper semicontinuous function. Then the Gaussian curvature of $ds^2$ on $\Delta\setminus\{g=0\}$ is defined by
\begin{align}\label{gc}
K(ds^2):=-\frac{1}{2g}\Delta\log g.
\end{align}
Here $\Delta u(\zeta):=4\liminf\limits_{r\rightarrow0}\frac{1}{r^2}\left\{\frac{1}{2\pi}\int_0^{2\pi}u(\zeta+re^{i\theta})d\theta-u(\zeta)\right\}$ for an upper semicontinuous function $u$.
\end{defn}

According to \cite{finsler}, we can define the holomorphic sectional curvature of a Finsler metric as follows.
\begin{defn}

Let $F$ be a Finsler metric on domain $\Omega\subset\mathbb{C}^n$ and $G=F^2$. For $p\in\Omega$ and $0\neq v\in T^{1,0}_z\Omega$, the holomorphic sectional curvature is defined by
\[
H(F)(p,v):=\sup\left\{K(\varphi^*G(0))\right\}.
\]
Here the supremum ranges through all holomorphic mappings $\varphi:\Delta\rightarrow\Omega$, satisfying $\varphi(0)=p$ and $\varphi'(0)=\lambda v$ with some $\lambda\in\mathbb{C^*}$, and $K(\varphi^*G)$ is the Gaussian curvature of pseudo-metric $\varphi^*G$ on $\Delta$.
\end{defn}

\begin{rmk}\label{hsc}
If the smooth Finsler metric $F$ is a Hermitian metric, then two definitions above of holomorphic sectional curvature coincide. It is worth noting that for any $(p,v)$, if $F$ is the Kobayashi metric or \C\ metric on a Kobayashi hyperbolic or \C\ hyperbolic domain $\Omega$, then $H(F)(p,v)\geq-4$ or $H(F)(p,v)\leq-4$ (see \cite{suzuki,wong} for more details). Moreover, if $\Omega$ is a bounded convex domain, then Kobayashi metric coincides with \C\ metric, which implies that $H(F)(p,v)\equiv-4$.

\end{rmk}

\subsection{Quasi-finite geometry}\hfill

Now we recall the notion of quasi-finite geometry appeared in \cite{ty}, which is an important tool in \K\ geometry.
\begin{defn}\label{qb}
	Suppose $(M^n,g)$ is a K$\ddot{\operatorname{a}}$hler manifold. We say that $(M^n,g)$ has $quasi$-$finite\ geometry\ of\ order\ \alpha\geq0$, if there exist constants $r_2>r_1>0$ such that: for any point $z\in M^n$ there exists a neighbourhood $U\subset\mathbb{C}^n$ and a nonsingular holomorphic mapping  $\varphi:U\rightarrow M^n$ such that

(1) $\varphi(0)=z$,

(2) $B_{r_1}(0)\subset U\subset B_{r_2}(0)$,

(3) there exists a constant $C\geq1$ determined only by $r_1,r_2,n$ such that 
\[
\dfrac{1}{C}g_0\leq\varphi^{*}g\leq Cg_0,
\]
 where $g_0$ is the standard Euclidean metric on $\mathbb{C}^n$,

(4) for any integer $\alpha\geq q\geq0$, there exists a constant $A_q$ determined by $q,r_1,r_2,n$ such that
\begin{equation}\label{qbg}
\sup_{x\in U}\left|\dfrac{\partial^{|\mu|+|\nu|}(\varphi^{*}g)_{i\bar{k}}}{\partial z^{\mu}\bar{\partial}z^{\nu}}(x)\right|\leq A_q\ \operatorname{for\ all} |\mu|+|\nu|\leq q \operatorname{and} 1\leq i,k\leq n,
\end{equation}
where $(\varphi^{*}g)_{i\bar{k}}$ is the component of $\varphi^{*}g$ on $U$ in terms of the natural coordinates $z=(z_1,\cdots,z_n)$ and $\mu,\nu$ are the multiple indices with $|\mu|=\mu_1+\cdots+\mu_n$.

\end{defn}

\section{Proofs of main theorems}\label{sec3}

In this section, we prove our main theorems. For a smoothly bounded pseudoconvex domain of finite type $\Omega\subset\mathbb{C}^2$, we will construct the scaling sequence with respect to a boundary point of $\Omega$. We firstly recall the notion of normal convergence.
\begin{defn}\label{nc}
Suppose $\left\{\Omega_n\right\}\subset\mathbb{C}^n$ is a sequence of domains, we say that $\Omega_n$ $converges\ normally$ to a domain $\Omega_{\infty}\subset\mathbb{C}^n$ if the following hold:

(1) for any compact set $K$, if $K$ is contained in the interior of $\bigcap_{i>k}\Omega_i$ for some constant $k$, then $K\subset\Omega_{\infty}$,

(2) for any compact set $K_1\subset\Omega_{\infty}$, there exists a constant $k_1$ such that $K_1\subset\bigcap_{i>k_1}\Omega_i$.
\end{defn}

Suppose $\Omega\subset\mathbb{C}^2$ is a smoothly bounded pseudoconvex domain of finite type, and $p\in\partial\Omega$ is a boundary point of $m$-type. Now we recall the scaling sequence with respect to boundary point $p$. Without loss of generality, we may assume that $\Omega=\left\{(z,w)\in\mathbb{C}^2:\rho(z,w)<0\right\}$, where $\rho(z,w)$ is a smooth function and $p=(0,0)$. Then after changing of coordinates, we may assume that
\[
\rho(z,w)=\operatorname{Re}w+H_m(z)+o\left(|z|^{m+1}+|z||w|	\right)
\]
in a neighbourhood of $p$, where $H_m(z)$ is a homogeneous polynomial of degree $m$ from $\mathbb{C}$ to $\mathbb{R}$, subharmonic and without harmonic terms. 

Suppose $\{p_n\}_{n\in\mathbb{N}}\subset\Omega$ is a sequence that converges to $p$, and for each $n$ we consider constant $\epsilon_n>0$ such that $\left(p_n^{(1)},p_n^{(2)}+\epsilon_n\right)\in\partial\Omega$ with $p_n=\left(p_n^{(1)},p_n^{(2)}\right)$. Let 
\[
\psi_n^{-1}(z,w):=\left(p_n^{(1)}+z,p_n^{(2)}+\epsilon_n+d_{n,0}w+\sum_{k=2}^md_{n,k}z^k\right)
\]
be  an automorphism of $\mathbb{C}^2$, where $d_{n,k}$ are constants such that 
\[
\rho\circ\psi_n^{-1}(z,w)=\operatorname{Re}w+P_n(z)+o\left(|z|^{m+1}+|z||w|\right).
\]
Here $P_n(z):\mathbb{C}\rightarrow\mathbb{R}$ is a subharmonic polynomial without harmonic terms and $P_n(0)=0$ with degree $m$. Furthermore, we select constant $\tau_n>0$ such that 
\[
||P_n(\tau_n\cdot)||=\epsilon_n,
\]
where $||\cdot||$ denote the norm in the space of polynomials with degree at most $m$. 

For each $n\geq1$, we define the following mapping 
\[
\delta_n(z,w):=\left(\frac{z}{\tau_n},\frac{w}{\epsilon_n}\right)
\]
and $\alpha_n:=\delta_n\circ\psi_n$, which is an automorphism of $\mathbb{C}^2$. Note that if we let $\Omega_n=\alpha_n(\Omega)$ for each $n=1,2,\cdots$, then $\Omega_n$ converges to $\Omega_{\infty}:=\left\{(z,w)\in\mathbb{C}^2:\operatorname{Re}w+P_{\infty}(z)<0\right\}$ in the sense of Definition \ref{nc}. Here $P_{\infty}(z):\mathbb{C}\rightarrow\mathbb{R}$ is a real-valued subharmonic polynomial without harmonic terms and its degree is $m$. Moreover, we have 
\begin{equation}\label{not}
q_n:=\alpha_n(p_n)=\left(0,-d_{n,0}^{-1}\right)\rightarrow(0,-1)=q_{\infty},
\end{equation}
and the domain $\Omega_{\infty}$ is called a $model\ domain$.
\medskip

Under the scaling sequence above, we obtain the following result about stability of Kobayashi metrics.

\begin{lemmaNoParens}[{\cite[Lemma 5.2]{ver}}]\label{stability}
For $(z,v)\in\Omega_{\infty}\times\mathbb{C}^2$, we have 
\begin{equation}\label{stab}
\lim\limits_{n\rightarrow\infty}K_{\Omega_n}(z,v)=K_{\Omega_{\infty}}(z,v)
\end{equation}
and the convergence is uniform on compact subsets of $\Omega_{\infty}\times\mathbb{C}^2$.
\end{lemmaNoParens}

By considering the extremal mappings of Carath$\acute{\operatorname{e}}$odory metric and using Montel's theorem,  the stability of the family $\left\{C_{\Omega_n}\right\}$  holds as follows.
\begin{lemma}\label{cara}
For $(z,v)\in\Omega_{\infty}\times\mathbb{C}^2$, then
\[
\limsup\limits_{n\rightarrow\infty}C_{\Omega_n}(z,v)\leq C_{\Omega_{\infty}}(z,v).
\]
\end{lemma}

Although we can not obtain the stability of $\left\{C_{\Omega_n}\right\}$ like $\left\{K_{\Omega_n}\right\}$, by considering boundary behaviour of Kobayashi metric and \C\ metric, we can get the following lemma. 

\begin{lemmaNoParens}\label{ca}
If $\lim\limits_{z\rightarrow p}\frac{C_{\Omega}(z,v)}{K_{\Omega}(z,v)}=1$ holds uniformly for $0\neq v\in\mathbb{C}^2$. Then
\[
\lim\limits_{n\rightarrow\infty}C_{\Omega_n}(z,v)=C_{\Omega_{\infty}}(z,v)
\]
holds uniformly on compact sets of $\Omega_{\infty}\times\mathbb{C}^2$.

\end{lemmaNoParens}

\begin{proof}
After taking a subsequence, we may assume that there exists a constant $\epsilon>0$ such that
\[
\left|C_{\Omega_j}(z_j,v_j)-C_{\Omega_{\infty}}(z_j,v_j)\right|>\epsilon
\]
holds for some $(z_j,v_j)$. Here each $(z_j,v_j)$ lies in a compact set $K\subset\Omega_{\infty}\times\mathbb{C}^2$. And we may assume that $z_j\rightarrow z_0,v_j\rightarrow v_0$. Without loss of generality, we suppose that
\[
\left|C_{\Omega_j}(z_j,v_j)-C_{\Omega_{\infty}}(z_0,v_0)\right|>\epsilon/2.
\]
Hence, there exists a constant $c>0$, independently of $j$, such that
\[
\left|\frac{C_{\Omega_j}(z_j,v_j)}{C_{\Omega_{\infty}}(z_0,v_0)}-1\right|>\epsilon/c.
\]
Since $\lim\limits_{z\rightarrow p}\frac{C_{\Omega}(z,v)}{K_{\Omega}(z,v)}=1$ holds uniformly for $0\neq v\in\mathbb{C}^2$, then for any $(z,v)\in\Omega_{\infty}\times\mathbb{C}^2$, we have $C_{\Omega_{\infty}}(z,v)=K_{\Omega_{\infty}}(z,v)$. To see this, we may assume $0\neq v$ and observe that 
\[
\limsup\limits_{n\rightarrow\infty}\dfrac{C_{\Omega_n}(z,v)}{K_{\Omega_n}(z,v)}\leq \dfrac{C_{\Omega_{\infty}}(z,v)}{K_{\Omega_{\infty}}(z,v)}\leq1
\]
from Lemma \ref{stability} and Lemma \ref{cara}. Note that $\li\frac{C_{\Omega}(z,v)}{K_{\Omega}(z,v)}=1$ uniformly for $0\neq v$ implies that for any $\epsilon_1>0$, there exists a constant $\delta>0$ such that
\[
 1-\epsilon_1<\frac{C_{\Omega}(z,v)}{K_{\Omega}(z,v)}< 1+\epsilon_1
\]
for any $0\neq v$ and $z\in\Omega$ with $\delta_{\Omega}(z)<\delta$. Since $\alpha_n^{-1}(z)$ converges to $p\in\partial\Omega$, hence for large $n$ we obtain $\delta_{\Omega}(\alpha_n^{-1}(z))<\delta$. Then 
\[
\dfrac{C_{\Omega_n}(z,v)}{K_{\Omega_n}(z,v)}=\dfrac{C_{\Omega}(\alpha_n^{-1}(z),(d\alpha_n^{-1})v)}{K_{\Omega}(\alpha_n^{-1}(z),(d\alpha_n^{-1})v)}>1-\epsilon_1
\]
for large $n$. Taking limit for $n$ implies that $C_{\Omega_{\infty}}(z,v)=K_{\Omega_{\infty}}(z,v)$.
Therefore, we can deduce that
\[
\lim\limits_{j\rightarrow\infty}\frac{C_{\Omega_j}(z_j,v_j)}{C_{\Omega_{\infty}}(z_0,v_0)}=\lim\limits_{j\rightarrow\infty}\frac{C_{\Omega_j}(z_j,v_j)}{K_{\Omega_{\infty}}(z_0,v_0)}=\lim\limits_{j\rightarrow\infty}\frac{C_{\Omega_j}(z_j,v_j)}{K_{\Omega_j}(z_j,v_j)}\frac{K_{\Omega_j}(z_j,v_j)}{K_{\Omega_{\infty}}(z_0,v_0)}=1,
\]
which is a contracdiction.
\end{proof}

\begin{rmk}\label{model}
If  $\lim\limits_{z\rightarrow p}\frac{C_{\Omega}(z,v)}{K_{\Omega}(z,v)}=1$ holds uniformly for $0\neq v\in\mathbb{C}^2$, then from the proof of Lemma \ref{ca} and Remark \ref{hsc}, we know that the holomorphic sectional curvature of $K_{\Omega_{\infty}}$ equals -4.
\end{rmk}

According to the notion of quasi-finite geometry, we obtain the following result, which concerns the convergence of \K\ metrics under the scaling sequence. The proof is similar with \cite[Proposition 6.1]{pin}, we provide it for completeness.

\begin{proposition}\label{ka}
Suppose $g_n$ is a complete K$\ddot{a}$hler metric on $\Omega_n$ for each $n$. If

(1) there exists a constant $A\geq 1$, independently on $n$, such that
\[
\frac{1}{A}K_{\Omega_n}(z,v)\leq \sqrt{g_n(z)(v,v)}\leq AK_{\Omega_n}(z,v)
\]
for all $z\in\Omega_n$ and $v\in\mathbb{C}^2$,

(2) each $(\Omega_n,g_n)$ has quasi-finite geometry of order $\alpha\geq 1$, where the constants appeared in Definition \ref{qb} are independent on $n$.

Then after passing to a subsequence, there exists a complete $C^{\alpha-1}$ smooth metric $g_{\infty}$ on $\Omega_{\infty}$ such that $g_n$ converges to $g_{\infty}$ locally uniformly in $C^{\alpha-1}$ topology.
\end{proposition}

\begin{proof}
In this proof, the distance induced by Kobayashi metric on $\Omega$ is denoted by $d_{\Omega}$. It is enough to  show that: for any compact set $K\subset\Omega_{\infty}$ and multi-indices $\mu,\nu$ with $|\mu|+|\nu|\leq\alpha$, there exist constants $J\geq0,C(K,\mu,\nu)\geq1$ depending only on $K,\mu,\nu$, such that 
\[
\sup_{n\geq J}\max_{x\in K}\left|\dfrac{\partial^{|\mu|+|\nu|}(g_n)_{i\bar{k}}}{\partial z^{\mu}\bar{\partial}z^{\nu}}(x)\right|\leq C(K,\mu,\nu).
\]

We suppose that there exists a compact set $K\subset\Omega_{\infty}$ and $\mu,\nu$ with $|\mu|+|\nu|\leq\alpha$ such that: for each integer $l\geq0$, there exists a constant $n_l\geq l$ and $x_l\in K$ such that 
\[
\left|\dfrac{\partial^{|\mu|+|\nu|}(g_{n_l})_{i\bar{k}}}{\partial z^{\mu}\bar{\partial}z^{\nu}}(x_l)\right|\geq l.
\]

Without loss of generality, we assume that $x_l\rightarrow x_{\infty}\in\Omega_{\infty}$ and $K\subset\Omega_{n_l}$ for all $l\geq1$.  From Difinition \ref{qb}, we know that for each $l$, there is a domain $U_l\subset\mathbb{C}^2$ and a nonsingular mapping $\varphi_l:U_l\rightarrow\Omega_{n_l}$ with $\varphi_l(0)=x_l$ satisfying conditions in Definition \ref{qb}, where the constants in Definition \ref{qb} are independent on $l$. 

Suppose $r<r_1$ is a fixed constant, then 
\begin{align*}
\sup_{w\in\varphi_l(B_r(0))}d_{\Omega_{n_l}}(x_l,w)\leq\sup_{\zeta\in B_r(0)}d_{B_{r_1}(0)}(0,\zeta)= d_{\mathbb{B}^2}\left(0,\frac{r}{r_1}\right).
\end{align*}
If $w\in\varphi_l(B_r(0))$, then we have that 
\[
d_{\Omega_{n_l}}(q_{n_l},w)\leq d_{\Omega_{n_l}}(q_{n_l},z_l)+d_{\Omega_{n_l}}(z_l,w)\leq d_{\Omega_{n_l}}(q_{n_l},z_l)+d_{\mathbb{B}^2}\left(0,\frac{r}{r_1}\right)\leq M
\]
for some constant $M>0$, independently of $l$. Hence from \cite[Lemma 5.6]{ver} each $\varphi_l(B_r(0))$ is compactly contained in $\Omega_{\infty}$  for all $l$ large. And it implies that $\varphi_l(B_r(0))$ is uniformly bounded.

 From condition (3) in Definition \ref{qb}, we obtain that 
\[
\dfrac{1}{C}|v|\leq\sqrt{g_{n_l}(x_l)((d\varphi_l)_0(v),(d\varphi_l)_0(v))}
\]
for each $l$ and $v\in\mathbb{C}^2$. It implies that 
\[
\dfrac{1}{C}|v|\leq\sqrt{g_{n_l}(x_l)((d\varphi_l)_0(v),(d\varphi_l)_0(v))}\leq AK_{\Omega_{n_l}}(x_l,(d\varphi_l)_0(v))\leq A\dfrac{|(d\varphi_l)_0(v)|}{\delta_{\Omega_{n_l}}(x_l)}
\]
and
\[
|(d\varphi_l)_0(v)|\geq \dfrac{\delta_{\Omega_{n_l}}(x_l)}{AC}|v|\rightarrow\dfrac{\delta_{\Omega_{\infty}}(x_{\infty})}{AC}|v|.
\]
Then Montel's theorem implies that $\varphi:=\lim\limits_{l\rightarrow\infty}\varphi_l:B_{r_1}(0)\rightarrow\Omega_{\infty}$ is holomorphic and nonsingular at 0.

Note that $\varphi|_{U_1}$ is invertible in some neighbourhood $U_1$, and $\varphi_l$ converges $\varphi$ in $C^{\infty}$ topology.  Then  $\varphi_l|_{U_2}$ is invertible for large $l$ in some neighbourhood $U_2\subset U_1$. Moreover, we know that $\left(\varphi_l|_{U_2}\right)^{-1}$ converges locally uniformly to $\left(\varphi|_{U_2}\right)^{-1}$.

After choosing a neighbourhood $V$ of $x_{\infty}\in\Omega_{\infty}$ such that $V\subset\varphi(U_2)$, we may assume that $V\subset\varphi_l\left(U_2\right)$ with $x_l\in V$ for large $l$ and 
\[
\sup_{l\geq N}\sup_{x\in V}\left|\dfrac{\partial^{|a|+|b|}(\varphi_l|_{U_2})^{-1}}{\partial z^a\bar{\partial}z^b}(x)\right|<\infty
\]
whenever $|a|+|b|\leq|\mu|+|\nu|$. Combining with  condition (4) in Definition \ref{qb}, we obtain that 
\[
\sup_{l\geq N}\left|\dfrac{\partial^{|\mu|+|\nu|}(g_{n_l})_{i\bar{k}}}{\partial z^{\mu}\bar{\partial}z^{\nu}}(x_l)\right|<\infty,
\]
which is a contracdiction.

Note that after passing to a subsequence, we may assume that $g_n$ converges locally uniformly to a $C^{\alpha-1}$ smooth pseudo-metric $g_{\infty}$ on $\Omega_{\infty}$ in $C^{\alpha-1}$ topology. From Lemma \ref{stab}, we have
\[
\sqrt{g_{\infty}(z)(v,v)}=\lim\limits_{n\rightarrow\infty}\sqrt{g_n(z)(v,v)}\geq\dfrac{1}{A}\lim\limits_{n\rightarrow\infty}K_{\Omega_n}(z,v)=\dfrac{1}{A}K_{\Omega_{\infty}}(z,v)>0
\]
and
\[
\sqrt{g_{\infty}(z)(v,v)}=\lim\limits_{n\rightarrow\infty}\sqrt{g_n(z)(v,v)}\leq{A}\lim\limits_{n\rightarrow\infty}K_{\Omega_n}(z,v)={A}K_{\Omega_{\infty}}(z,v),
\]
then $g_{\infty}$ is a complete $C^{\alpha-1}$ smooth metric on $\Omega_{\infty}$. And it completes the proof.
\end{proof}

\noindent $Proof\ of\ Theorem\ \ref{kob}$: We just need to prove $(2)\Rightarrow (1)$ and $(3)\Rightarrow(1)$.
\medskip

$(3)\Rightarrow(1)$: Suppose $\Omega$ satisfies condition (2) in Theorem \ref{kob}, then it is covered by $\mathbb{B}^2$ or $\mathbb{C}^2$ or $\mathbb{CP}^2$. Since covering map is a local isometry for Kobayashi metrics and $\Omega$ is Kobayashi hyperbolic, then $\Omega$ is covered by $\mathbb{B}^2$.
\medskip

$(2)\Rightarrow(1)$: Suppose $(\Omega,K_{\Omega})$ has quasi-finite geometry of order three, then each $(\Omega_n,K_{\Omega_n})$ has quasi-finite geometry of order three, where the constants in Definition \ref{qb} are independently on $n$. In particular, the holomorphic sectional curvature of $K_{\Omega_n}$ is uniformly bounded.
\medskip

\textbf{Claim}:$\lim\limits_{n\rightarrow\infty}H(K_{\Omega_n})(z,v)=H(K_{\Omega_{\infty}})(z,v)=-4$ holds uniformly on compact sets of $\Omega_{\infty}\times(\mathbb{C}^2)^*$.
\medskip

\noindent{}$Proof\ of\ claim$: Since $H(K_{\Omega_n})(z,v)\geq-4$ for each $n$ and $(z,v)\in\Omega_n\times(\mathbb{C}^2)^*$, we need to show that
\[
\limsup_{n\rightarrow\infty}H({K_{\Omega_n}})(z,v)\leq-4
\]
holds uniformly on compact sets of $\Omega_{\infty}\times(\mathbb{C}^2)^*$. For a contradiction, we assume there exist $(z_j,v_j)$ lying in a compact set of $\Omega_{\infty}\times(\mathbb{C}^2)^*$ with $|v_j|=1$, such that 
\[
\lim\limits_{j\rightarrow\infty}H({K_{\Omega_j}})(z_j,v_j)=c>-4.
\]
We assume that $z_j\rightarrow z_0\in\Omega_{\infty}$ and $v_j\rightarrow v_0$ with $|v_0|=1$. From Proposition \ref{ka} and Lemma \ref{stability}, we know that after passing to a subsequence, $K_{\Omega_j}$ converges to $K_{\Omega_{\infty}}$ locally uniformly in $C^2$ topology. And it means that $H({K_{\Omega_j}})(z_j,v_j)\rightarrow H({K_{\Omega_{\infty}}})(z_0,v_0)=-4$, and this is a contradiction. 
\medskip

 Hence we obtain that 
\[
\li H({K_{\Omega}})(z,v)=-4
\]
uniformly for $0\neq v\in\mathbb{C}^2$. Theorem 1.2 in \cite{str} means that $\Omega$ is strongly pseudoconvex, and we complete the proof since Theorem \ref{zi}.
\medskip

\noindent $Proof\ of\ Theorem\ \ref{kem}$: Suppose $\Omega$ is covered by the unit ball. Since holomorphic covering map is a local imsometry of Kobayashi metrics and \K-Einstein metrics, then \K-Einstein metric is a scalar multiple of Kobayashi metric on $\Omega$. 

On the other hand, from \cite[Proposition 3.11]{str} we know that for any $x\in\Omega$, there exists a biholomorphism $\varphi_x$ from $B_r(0)$ to $\varphi_x(B_r(0))\subset\Omega$ satisfying

(1) $\varphi_x(0)=x,$

(2) $\frac{1}{A}g_0\leq\varphi_x^*B_{\Omega}\leq Ag_0$ with some constant $A\geq1$.

\noindent{}Here $r,A$ are independent on $x$ and $g_0$ is the standard Euclidean metric on $\mathbb{C}^2$.

 Since $B_{\Omega}$ is equivalent to the \K-Einstein metric $g_{KE}$ on $\Omega$, after a similar proof of \cite[Lemma 3]{us} in coordinate $\left(B_r(0),\varphi_x\right)$, then $(\Omega,g_{KE})$ has quasi-finite geometry of order infinity. Thus if $(\Omega,g_{KE})$ satisfies condition (2) in Corollary \ref{kem}, then Theorem \ref{kob} implies that $\Omega$ is biholomorphic to a ball quotient.

\medskip

\noindent $Proof\ of\ Theorem\ \ref{car}$: It's enough to prove $(2)\Rightarrow (1)$ and $(3)\Rightarrow(1)$.
\medskip

$(3)\Rightarrow(1)$: For a bounded domain $\Omega$, we know that $K_{C_{\Omega}}(z,v)\leq -4$ for any $(z,v)\in\Omega\times(\mathbb{C}^2)^*$. If $C_{\Omega}$ satisfies condition (3) in Theorem \ref{car}, then it has constant negative  holomorphic sectional curvature. And \cite[Theorem 1.2]{str} implies that $\Omega$ is strongly pseudoconvex. We know that $\Omega$ is biholomorphically equivalent to a ball  from Theorem 1.9 in \cite{dww}.
\medskip

$(2)\Rightarrow(1)$: If $C_{\Omega}$ satisfies condition (2) in Theorem \ref{car}, then Lemma \ref{ca} and Proposition \ref{ka} impliy that 
\[
\li H({C_{\Omega}})(z,v)=-4
\]
uniformly for $0\neq v\in\mathbb{C}^2$ after a similar argument of proof of Theorem \ref{kob}. Hence $\Omega$ is strongly pseudoconvex from \cite[Theorem 1.2]{str}, and this completes the proof of Theorem \ref{car}.
\medskip

Now we prove Theorem \ref{bc}, and it can be confirmed by the following proposition and Theorem 1.10 in \cite{dww}.

\begin{proposition}
Let $\Omega\subset\mathbb{C}^2$ be a smoothly bounded pseudoconvex domain of finite type, and $p\in\partial\Omega$ be a boundary point. If 
\[
\limsup_{z\rightarrow p}\frac{H(B_{\Omega})}{L^2(\Omega)}\leq-4
\]
and 
\[
\lim\limits_{z\rightarrow p}\frac{C_{\Omega}(z,v)}{K_{\Omega}(z,v)}=1
\]
uniformly for $0\neq v\in\mathbb{C}^2$, then $\Omega$ is strongly pseudoconvex at $p$.

\end{proposition}

\begin{proof}
Let $\left\{\Omega_n\right\}$ be the scaling sequence with respect to $p$, and $\Omega_{\infty}$ be the corresponding model domain. Suppose $B_{\Omega_n}$ is the Bergman metric on $\Omega_n$, then Proposition 3.10 in \cite{str} implies that there exists a complete \K\ metric $B_{\infty}$ on $\Omega_{\infty}$ such that $B_n$ converges to $B_{\infty}$ on $\Omega_{\infty}$ locally uniformly in $C^{\infty}$ topology after taking a subsequence. Note that $H(B_{\infty})\leq -4L^2(\Omega)$ holds in $\Omega_{\infty}$. Since $C_{\Omega}$ and $B_{\Omega}$ are biholomorphically invariant and $C_{\Omega}\leq L(\Omega)B_{\Omega}$, then 
\[
C_{\Omega_{\infty}}(z,v)\leq L(\Omega)B_{\infty}(z,v)
\]
holds for any $(z,v)\in\Omega_{\infty}\times\mathbb{C}^2$ from Lemma \ref{ca}. Ahlfors-Schwarz lemma implies that
\[
K_{\Omega_{\infty}}(z,v)\geq L(\Omega)B_{\infty}(z,v)
\]
for any $(z,v)\in\Omega_{\infty}\times\mathbb{C}^2$. Since $K_{\Omega_{\infty}}=C_{\Omega_{\infty}}$, then $K_{\Omega_{\infty}}=L(\Omega)B_{\infty}=C_{\Omega_{\infty}}$, and it implies that $\Omega_{\infty}$ is biholomorphic to the unit ball. Then we obtain that $\Omega$ is strongly pseudoconvex at $p$.
\end{proof}

\bibliography{ref}
\bibliographystyle{plain}{}
\end{document}